\documentclass[11pt]{article}

\usepackage{subfigure}
\usepackage{setspace}
\usepackage{titletoc}
\usepackage{titlesec}
\usepackage{algorithmicx}
\usepackage{subfigure}
\usepackage{subcaption}
\usepackage{mathtools}
\usepackage{emptypage}
\usepackage{amsfonts,amssymb}
\usepackage{subfigure}
\usepackage{enumerate}
\usepackage{graphicx}
\usepackage{epstopdf}
\usepackage{etoolbox}
\usepackage{titlesec}
\usepackage[titles,subfigure]{tocloft}
\usepackage{cite}
\usepackage{mathrsfs}
\usepackage{listings}
\usepackage{amsmath}
\usepackage{tikz-cd}
\usepackage{pdfpages}
\usepackage{amsthm}

\newtheorem{theorem}{Theorem}[section]

\newtheorem{proposition}[theorem]{Proposition}

\theoremstyle{definition}

\theoremstyle{remark}

\usepackage{amsthm}
\usepackage{geometry}
\usepackage{fancyhdr}
\usepackage{lastpage}
\usepackage{afterpage}

\usepackage[colorlinks=true,linkcolor=blue,citecolor=blue]{hyperref}
\numberwithin{equation}{section}
\begin{document}
\noindent
{\Large \textbf{Some Eigenvalue Inequalities for the Schrödinger Operator on Integer Lattices}}\vspace{1\baselineskip}\par\noindent
{\large \textbf{Wentao Liu} }
\vspace{1\baselineskip}\par
\noindent
{\large \textbf{Abstract}}:
In this paper, we establish analogues of the Payne-Pólya-Weinberger, Hile-Protter,
and Yang eigenvalue inequalities for the Schrödinger operator on arbitrary finite
subsets of the integer lattice $\mathbb{Z}^n$. The results extend known inequalities
for the discrete Laplacian to a more general class of Schrödinger operators with
nonnegative potentials and weighted eigenvalue problems.
\vspace{0.5\baselineskip}\par\noindent
\par\noindent
\section{Introduction}
\noindent
There is a vast literature on eigenvalue problems for the Laplacian operator with
Dirichlet boundary conditions; see, for example, \cite{payne1955ratio,thompson1969ratio,hile1980inequalities,chen2016estimates}.
In the classical eigenvalue problem, let $\Omega \subset \mathbb{R}^n$ be a bounded
domain. We consider
\[
\begin{cases}
-\Delta u = \lambda u, & \text{in } \Omega, \\
u = 0, & \text{on } \partial \Omega.
\end{cases}
\]
\par\noindent
In 1956 Payne,Polya, and Weinberger \cite{payne1955ratio} proved the following universal inequalities for the $\lambda_i's$
in the case when $n=2$:
\begin{align*}
\lambda_{k+1}-\lambda_{k}\le \frac{2}{k} \sum_{i=1}^{k}\lambda_i,
\end{align*}
and can be extend to the case of general $n$ one has the following inequailty:
\begin{align*}
  \lambda_{k+1}-\lambda_{k}\le \frac{4}{nk} \sum_{i=1}^{k}\lambda_i,
\end{align*}
this inequality is known as the Payne-Pólya-Weinberger (PPW) inequality.
A stronger inequality was derived in 1980 by Hile and Protter \cite{hile1980inequalities}:
\begin{align*}
\sum_{i=1}^{k}\frac{\lambda_i}{\lambda_{k+1}-\lambda_k}\ge \frac{nk}{4},
\end{align*}
and a more sharp inequality was proved by Yang \cite{yang1991estimate}:
\begin{align*}
\sum_{i=1}^{k}(\lambda_{k+1}-\lambda_i)(\lambda_{k+1}-(1+\frac{4}{n})\lambda_i)\le0,
\end{align*}
this inequailty will be called as the Yang's first inequality.And it can be proved that
\begin{align*}
\lambda_{k+1}\le \frac{1}{k}(1+\frac{4}{n})\sum_{i=1}^{k}\lambda_i,
\end{align*}
which we called this Yang's second inequailty.
Moreover, Yang's first inequality implies Yang's second inequality, which in turn
implies the Hile-Protter inequality, and hence the PPW inequality:
\begin{align*}
Yang 1 \Rightarrow Yang 2 \Rightarrow HP \Rightarrow PPW.
\end{align*}\par
In the paper \cite{hua2023payne}, the authors considered eigenvalue problems for the
Dirichlet Laplacian on integer lattices. They established analogues of the
inequalities mentioned above, and under the condition $\lambda_k \le 1 + \frac{4}{n}$,
they also obtained relation (1.6).

In the present paper, we consider a more general class of operators than the
Laplacian. In particular, we study Schrödinger operators of the form
$-\Delta + V(x)$ with nonnegative potential $V \ge 0$ on $\Omega$, as well as
weighted eigenvalue problems. Specifically, we consider the eigenvalue problem
\[
\begin{cases}
-\Delta u + V u = \lambda \rho u, & \text{in } \Omega, \\
u = 0, & \text{on } \partial \Omega,
\end{cases}
\]
where the density function $\rho(x)$ is assumed to be positive.

A substantial body of literature has been developed on analysis on graphs,
 including the study of the graph Laplacian and its associated analytic and geometric properties;
  see, for example, \cite{grigor2018introduction,dodziuk1984difference,chung2000harnack,chung2000weighted}.
First, we recall some basic notions concerning Laplacian operators on discrete
spaces, namely graphs. Let $(V,E)$ be a simple, undirected, locally finite graph
with vertex set $V$ and edge set $E$, and assume that the graph has no isolated
vertices. Two vertices $x,y \in V$ are said to be neighbors, denoted by $x \sim y$,
if there exists an edge connecting them. The degree of a vertex $x$, denoted by
$d_x$, is defined as the number of neighbors of $x$.

The discrete Laplacian $\Delta$ on $(V,E)$ is defined by
\[
\Delta f(x) := \frac{1}{d_x} \sum_{y \in V:\, y \sim x} \bigl( f(y) - f(x) \bigr),
\qquad \forall\, f : V \to \mathbb{R}.
\]
\noindent
Let $\Omega$ be a finite subset of $V$. We define the vertex boundary of $\Omega$ by
\[
\partial \Omega := \{ y \in V \setminus \Omega : y \sim x \text{ for some } x \in \Omega \}.
\]
\noindent
We denote by $\ell^{2}(\Omega,\rho)$ the Hilbert space of all real-valued functions
defined on $\Omega$, equipped with the inner product
\[
\langle f, g \rangle := \sum_{x \in \Omega} f(x)\, g(x)\, \rho(x)\, d_x,
\qquad f,g : \Omega \to \mathbb{R}.
\]
\noindent
We consider the eigenvalue problem for the Schrödinger operator
\[
\begin{cases}
-\Delta u + V u = \lambda \rho u, & \text{in } \Omega, \\
u = 0, & \text{on } \partial \Omega.
\end{cases}
\]
\noindent
For a finite subset $\Omega$, we denote by $\widetilde{H} = (-\Delta + V)/\rho$
the associated operator and write $H = -\Delta + V$. The spectrum of
$\widetilde{H}$ consists of $N = |\Omega|$ real eigenvalues, which we label as
\[
0 < \lambda_1 \le \lambda_2 \le \cdots \le \lambda_N.
\]

In the following, we shall derive more general eigenvalue inequalities for this
class of operators.
\par\noindent
\begin{theorem}\label{thm:1.1} (Yang's first type inequality).
Let $\Omega$ be a finite subset of $\mathbb{Z}^n$, and let $\lambda_i$ denote the $i$-th eigenvalue of the Schrödinger problem on $\Omega$.
Then, for any $1 \le k \le \sharp \Omega - 1$, one has
\begin{align}
  \sum_{i}^{k}(\lambda_{k+1}-\lambda_i)\Bigl(\lambda_{k+1}\Bigl(\frac{\rho_{min}}{\rho_{max}}-{\rho_{min}}\lambda_i\Bigr)-\lambda_i\Bigl(\frac{\rho_{min}}{\rho_{max}}-{\rho_{min}}\lambda_i+\frac{4}{n}\Bigr)\Bigr)\le 0.
  \end{align}
\end{theorem}
\noindent
\begin{theorem}\label{thm:1.2} (Yang's second type inequality).
Let $\Omega$ be a finite subset of $\mathbb{Z}^n$, and let $\lambda_i$ denote the $i$-th eigenvalue of the Schrödinger problem on $\Omega$.
Then, for any $1 \le k \le \sharp \Omega - 1$, one has
\begin{align}
  \lambda_{k+1}\Bigl(1-\frac{\rho_{max}}{k}\sum_{i}^{k}\lambda_i\Bigr)\le \frac{1}{k}\Bigl(1+\frac{4\rho_{max}}{n\rho_{min}}\Bigr)\sum_{i}^{k}\lambda_i-\frac{\rho_{max}}{k}\sum_{i}^{k}\lambda_i^2.
  \end{align}
\end{theorem}
\noindent
\begin{theorem}\label{thm:1.3} (Hile--Protter type inequality).
Let $\Omega$ be a finite subset of $\mathbb{Z}^n$, and let $\lambda_i$ denote the $i$-th eigenvalue of the Schrödinger problem on $\Omega$.
Then, for any $1 \le k \le \sharp \Omega - 1$, one has
\begin{align}
  \sum_{i}^{k}\frac{\lambda_i}{\lambda_{k+1}-\lambda_i}\ge \frac{nk}{4}\Bigl(\frac{\rho_{min}}{\rho_{max}}-\frac{\rho_{min}}{k}\sum_{i}^{k}\lambda_i\Bigr).
  \end{align}
\end{theorem}
\noindent
\begin{theorem}\label{thm:1.4} (Payne--Pólya--Weinberger type inequality).
Let $\Omega$ be a finite subset of $\mathbb{Z}^n$, and let $\lambda_i$ denote the $i$-th eigenvalue of the Schrödinger problem on $\Omega$.
Then, for any $1 \le k \le \sharp \Omega - 1$, one has
\begin{align}
  (\lambda_{k+1}-\lambda_k)\Bigl(\frac{\rho_{min}}{\rho_{max}}-\frac{\rho_{min}}{k}\sum_{i}^{k}\lambda_i\Bigr)\le\frac{4}{nk}\sum_{i}^{k}\lambda_i.
  \end{align}
\end{theorem}
  \par
Moreover, when the density function satisfies $\rho(x) \equiv 1$, the above inequalities reduce to those obtained in \cite{hua2023payne}.

\section{Preliminaries}
\noindent
Throughout this section, we adopt the notation and conventions of \cite{hua2023payne}.
Let $G=(V,E)$ be a simple, undirected, and locally finite graph.
It is convenient to introduce the adjacency function
\begin{align*}
\mu : V \times V \to \{0,1\}, \qquad
\mu_{xy} =
\begin{cases}
1, & \text{if } x \sim y, \\
0, & \text{otherwise}.
\end{cases}
\end{align*}
With this notation, the degree of a vertex $x \in V$ is given by
\begin{align*}
d_x := \sum_{y \in V} \mu_{xy}.
\end{align*}
\noindent
We denote by $\mathbb{Z}^n$ the $n$-dimensional integer lattice graph, whose edge set is defined as
\begin{align*}
E
=
\bigl\{
\{x,y\} \subset \mathbb{Z}^n :
\sum_{i=1}^{n} |x_i - y_i| = 1
\bigr\}.
\end{align*}
\noindent
Given a function $f : \mathbb{Z}^n \to \mathbb{R}$ and two vertices $x,y \in \mathbb{Z}^n$,
we define the discrete gradient along the edge $(x,y)$ by
\begin{align*}
\nabla_{xy} f := f(y) - f(x).
\end{align*}
\par
Under this notation,we also have Green's formula on graphs,see,for example, \cite{grigor2018introduction}.\par\noindent
$\mathbf{Proposition\; 2.1}$ Let $(V,E)$ be a graph, $f$ be a function with finite support on $V$ and $g$ be any function
on $V$. Then
\begin{align}
\sum_{x\in V}(\Delta f)(x)g(x)d_x=-\frac{1}{2}\sum_{x,y \in V}\mu_{xy}\nabla_{xy}f\nabla_{xy}g=\sum_{x\in V}(\Delta g)(x)f(x)d_x.
\end{align}\par
We further introduce the following discrete analogue of the squared gradient of a function.
 \par\noindent
$\mathbf{ Definition\; 2.2}$ The gradient form\:$\Gamma$\:,which called the "carré du champ" operator,is defined by,for $f,g:V \rightarrow \mathbb{R} \;\text{and}\: x\in V$,
 \begin{align}\label{eq:2.2}
  \Gamma(f,g)(x)&=\frac{1}{2}(\Delta(fg)-f\Delta g-g\Delta f)(x)\notag\\&=\frac{1}{2d_x}\sum_{y}\mu_{xy} \nabla_{xy}f\nabla_{xy}g. \end{align}\par\noindent
 $\mathbf{Lemma \; 2.3}$ Let $u$ be a function on\:$ \mathbb{Z}^n$ \:of finite support and $\{x_\alpha\}_{\alpha =1}^{n}$be standard coordinate functions.Then\par\noindent

 \begin{align}\label{eq:2.3}
  \sum_{\alpha = 1}^{n}\frac{1}{2}\sum_{x,y}{}|\nabla_{xy}(x_\alpha)|^2|\nabla_{xy}u|^2\mu_{xy}=\sum_{x}{}\Gamma(u)(x)d_x,\end{align}\par\noindent

 \begin{align}\label{eq:2.4}
  \sum_{\alpha=1}^{n}\sum_{x}{}\Gamma(x_\alpha,u)^2(x)d_x\le\frac{1}{2n}\sum_{x}{}\Gamma(u)(x)d_x.\end{align}\par
 Details can be found in \cite{hua2023payne}.
 \section{Proof of the Main Results}
\noindent
In this section, we prove the main results by following the approach of
\cite{ashbaugh2002universal,hua2023payne}.\par
Let $\Omega$ be a finite subset of $\mathbb{Z}^n$ and $\lambda_k$ be the $k$-th eigenvalue of Schrödinger problem on $\Omega$. 
Let $k \le \sharp\Omega -1.$ For any $1 \le i \le k$, let $\{u_i\}_{i=1}^k$ be the normalized eigenfunctions associated with the first
$k$ eigenvalues $\lambda_i$, that is,
\begin{align*}
-\Delta u_i(x) + V(x) u_i(x) &= \lambda_i \rho(x) u_i(x), \\
\sum_{x \in \Omega} u_i(x) u_j(x)\,\rho(x)\, d_x &= \delta_{ij},
\qquad 1 \le i,j \le k.
\end{align*}
And it can be easily check that:
\begin{align*}
\sum_{x\in \mathbb{Z}^n}(-\Delta+V)(u_i)(x)u_i(x)d_x=\lambda_i.
\end{align*}
Throughtout this paper,we consider the zero extension of functions unless specified otherwise,because $V$ is 
nonnegative and Green's formula we have:
\begin{align}\label{eq:3.1}
\sum_{x\in \mathbb{Z}^n}\Gamma(u_i)(x)d_x\le \lambda_i.
\end{align}
Let $g$ be one of the coordinate functions, namely $g(x) = x_\alpha$ for some
$1 \le \alpha \le n$. By \eqref{eq:2.2} and the identity
$\Delta(g^2) = 2 g \Delta g + 2 \Gamma(g)$, we have
\begin{align}
\Delta g = 0, \qquad \Delta(g^2) = \frac{1}{n}.
\end{align}
For each $1 \le i \le k$, define
\begin{align*}
\phi_i := g u_i - \sum_{j=1}^k a_{ij} u_j,
\end{align*}
where
\begin{align*}
a_{ij} := \sum_{x \in \mathbb{Z}^n} g(x)\, u_i(x)\, u_j(x)\, \rho(x)\, d_x,
\end{align*}\par\noindent
by construction, $\phi_i$ is orthogonal to $\{u_j\}_{j=1}^k$ in $\ell^2(\Omega,\rho)$, namely,
\begin{align*}
\sum_{x \in \Omega} \phi_i(x) u_j(x)\, \rho(x)\, d_x = 0,
\qquad 1 \le j \le k.
\end{align*}
Moreover, it is clear that $a_{ij} = a_{ji}$ for all $1 \le i,j \le k$.
\par\noindent
Next, we introduce
\begin{align*}
b_{ij} := \sum_{x \in \Omega} u_j(x)\, \Gamma(g,u_i)(x)\, d_x.
\end{align*}
\par
We have the following identity.

\begin{proposition}
For all $1 \le i,j \le k$, it holds that
\begin{align}
2 b_{ij} = (\lambda_i - \lambda_j) a_{ij}.
\end{align}
\end{proposition}

\begin{proof}
By Green's formula, we compute
\begin{align*}
\lambda_j a_{ij}
&= \sum_{x \in \Omega} g(x)\, u_i(x)\, H u_j(x)\, d_x \notag \\
&= \sum_{x \in \Omega} g(x)\, u_i(x)\, (-\Delta u_j(x))\, d_x
   + \sum_{x \in \Omega} V(x)\, g(x)\, u_i(x)\, u_j(x)\, d_x \notag \\
&= - \sum_{x \in \Omega} u_j(x)
   \bigl( (g\Delta ( u_i))(x) + 2 \Gamma(g,u_i)(x) \bigr) d_x
   + \sum_{x \in \Omega} V(x)\, g(x)\, u_i(x)\, u_j(x)\, d_x \notag \\
&= \sum_{x \in \Omega}g(x) H u_i(x)\, u_j(x)\, d_x
   - 2 \sum_{x \in \Omega} u_j(x)\, \Gamma(g,u_i)(x)\, d_x \notag \\
&= \lambda_i a_{ij} - 2 b_{ij}.
\end{align*}
This completes the proof.
\end{proof}
\begin{proof}[Proof of Theorem 1.1]
By the Rayleigh quotient, we have
\begin{align}
\lambda_{k+1} \sum_{x \in \Omega} \phi_i^2(x) \rho(x) d_x 
\leq \langle \phi_i, \widetilde{H} \phi_i \rangle 
= \sum_{x \in \Omega} \phi_i(x) (H \phi_i)(x) d_x,
\end{align}
\noindent
expanding the right-hand side, we obtain
\begin{align}
\sum_{x \in \Omega} \phi_i(x) (H \phi_i)(x) d_x 
&= \sum_{x \in \Omega} \phi_i(x) \Bigl(-\Delta \Bigl(g u_i - \sum_{j} a_{ij} u_j\Bigr) + V \Bigl(g u_i - \sum_{j} a_{ij} u_j\Bigr)\Bigr)(x) d_x \notag \\
&= \sum_{x \in \Omega} \phi_i(x) \Bigl( g H u_i - 2 \Gamma(g, u_i) - \sum_j a_{ij} H u_j \Bigr)(x) d_x \notag \\
&= \lambda_i \sum_{x \in \Omega} \phi_i^2(x) \rho(x) d_x - 2 \sum_{x \in \Omega} \phi_i(x) \Gamma(g, u_i)(x) d_x,
\end{align}
\noindent
rearranging yields the inequality
\begin{align}\label{eq:3.6}
\lambda_{k+1} - \lambda_i 
\leq \frac{-2 \sum_{x \in \Omega} \phi_i(x) \Gamma(g, u_i)(x) d_x}{\sum_{x \in \Omega} \phi_i^2(x) \rho(x) d_x}.
\end{align}
Next, we compute the numerator term as
\begin{align}
0 \le -2 \sum_{x \in \Omega} \phi_i(x) \Gamma(g, u_i)(x) d_x 
= -2 \sum_{x \in \Omega} (g u_i \Gamma(g, u_i))(x) d_x + \sum_j (\lambda_i - \lambda_j) a_{ij}^2,
\end{align}
and denote
\[
I := 2 \sum_{x \in \Omega} (g u_i \Gamma(g, u_i))(x) d_x.
\]
By expanding \( I \), we have
\begin{align}\label{eq:3.8}
I 
&= \sum_{x,y} \mu_{xy} g(x) u_i(x) \nabla_{xy} g \, \nabla_{xy} u_i \notag \\
&= \frac{1}{2} \sum_{x,y} \mu_{xy} \big( g(x) u_i(x) + g(y) u_i(y) \big) \nabla_{xy} g \, \nabla_{xy} u_i \notag \\
&= \frac{1}{4} \sum_{x,y} \mu_{xy} \nabla_{xy}(g^2) \, \nabla_{xy}(u_i^2) + I_g(u_i) \notag \\
&= -\frac{1}{2} \sum_{x \in \Omega} \Delta(g^2)(x) u_i^2(x) d_x + I_g(u_i) \notag \\
&= -\frac{1}{2n} \sum_{x \in \Omega} u_i^2(x) d_x + I_g(u_i),
\end{align}
where the correction term \( I_g(u_i) \) is defined by
\[
I_g(u_i) := \frac{1}{4} \sum_{x,y} \mu_{xy} |\nabla_{xy} g|^2 |\nabla_{xy} u_i|^2.
\]
For the simplified expression \eqref{eq:3.8}, by \eqref{eq:2.3}, \eqref{eq:2.4}, \eqref{eq:3.1} and the definition of $u_i$, we have the following estimaties hold for the two relevant terms:
\begin{align}\label{eq:3.9}
\sum_{x}u_i(x)^2d_x \ge \frac{1}{\rho_{max}},\qquad \sum_{\alpha =1}^{n}I_{x_\alpha}(u_i) \le \frac{\lambda_i}{2}.
\end{align}
We also have
\begin{align}
\Bigl(-2\sum_{x\in \Omega}&\phi_i(x)\Gamma(g,u_i)d_x\Bigr)^2=\Bigl(-2\sum_{x\in\Omega}\phi_i(x)\Bigl(\Gamma(g,u_i)-\sum_{j}b_{ij}u_j\rho(x)d_x\Bigr)\Bigr)^2\notag\\
&\le 4\Bigl(\sum_{x\in\Omega}\phi_i^2(x)\rho(x)d_x\Bigr)\Bigl(\sum_{x\in\Omega}(\rho^{-1}(x)\Gamma(g,u_i)^2d_x-2\Gamma(g,u_i)\sum_{j}b_{ij}u_jd_x)+\sum_{j}b_{ij}^2\Bigr)\notag\\
&=4\Bigl(\sum_{x\in\Omega}\phi_i^2(x)\rho(x)d_x\Bigr)\Bigl(\sum_{x\in\Omega}\rho^{-1}(x)\Gamma(g,u_i)^2d_x-\sum_{j}b_{ij}^2\Bigr),
\end{align}
consequently\begin{align}\lambda_{k+1}-\lambda_{i}\le \frac{-2\sum_{x\in \Omega}\phi_i(x)\Gamma(g,u_i)d_x}{\sum_{x}\phi_i^2(x)\rho(x)d_x}
\le \frac{4\sum_{x\in\Omega}\rho^{-1}(x)\Gamma(g,u_i)^2d_x-\sum_{j}(\lambda_i-\lambda_j)^2a_{ij}^2}
{\frac{1}{2n}\sum_{x}u_i(x)^2d_x-I_g(u_i)+\sum_{j}(\lambda_i-\lambda_j)a_{ij}^2},\end{align}
equivently\begin{align}(\lambda_{k+1}-\lambda_{i})\Bigl(\frac{1}{2n}&\sum_{x}u_i(x)^2d_x-I_g(u_i)+\sum_{j}(\lambda_i-\lambda_j)a_{ij}^2\Bigr)
\notag\\&\le 4\sum_{x\in\Omega}\rho^{-1}(x)\Gamma(g,u_i)^2d_x-\sum_{j}(\lambda_i-\lambda_j)^2a_{ij}^2,\end{align}
\noindent
rewriting the terms involving \(a_{ij}\), we obtain\begin{align}\label{eq:3.13}
  (\lambda_{k+1}-\lambda_{i})\Bigl(\frac{1}{2n}\sum_{x}u_i(x)^2d_x-I_g(u_i)\Bigr)+\sum_{j}^{k}(\lambda_{k+1}-\lambda_{j})(\lambda_i-\lambda_j)a_{ij}^2
\notag\\\le 4\sum_{x\in \Omega}\rho^{-1}(x)\Gamma(g,u_i)^2d_x.\end{align}
Multiplying both sides by \((\lambda_{k+1}-\lambda_i)\) and summing over \(i=1,\dots,k\) eliminates all terms containing \(a_{ij}\), and yields
\begin{align}
\sum_{i}^{k}(\lambda_{k+1}-\lambda_i)^2\Bigl(\frac{1}{2n}\sum_{x}u_i(x)^2d_x-I_g(u_i)\Bigr)\le 4\sum_{i}^{k}(\lambda_{k+1}-\lambda_i)\sum_{x\in \Omega}\rho^{-1}(x)\Gamma(g,u_i)^2d_x.
\end{align}
We now proceed exactly by promoting $g$ to $x_\alpha$ and summing on $\alpha$ from $1$ to $n$ ,noting that \eqref{eq:3.1}, \eqref{eq:2.4} and \eqref{eq:3.9} to get
\begin{align}
    \sum_{i}^{k}(\lambda_{k+1}-\lambda_i)^2\Bigl(\frac{1}{2\rho_{max}}-\frac{1}{2}\lambda_i\Bigr)&\le 4\sum_{i}^{k}(\lambda_{k+1}-\lambda_i)\sum_{x\in \Omega}\rho^{-1}(x)\Gamma(g,u_i)^2d_x \notag \\
&\le \frac{2}{n\rho_{min}}\sum_{i}^{k}(\lambda_{k+1}-\lambda_i)\lambda_i,
\end{align}
that is
\begin{align*}
\sum_{i}^{k}(\lambda_{k+1}-\lambda_i)\Bigl(\lambda_{k+1}\Bigl(\frac{\rho_{min}}{\rho_{max}}-{\rho_{min}}\lambda_i\Bigr)-\lambda_i\Bigl(\frac{\rho_{min}}{\rho_{max}}-{\rho_{min}}\lambda_i+\frac{4}{n}\Bigr)\Bigr)\le 0.
\end{align*}
Thus we prove Theorem \ref{thm:1.1} \qedhere
\end{proof}
\begin{proof}[Proof of Theorem 1.2]
By \eqref{eq:3.6},we also have
\begin{align}
  \lambda_{k+1}-\lambda_{i}\le \frac{-2\sum_{x\in \Omega}\phi_i(x)\Gamma(g,u_i)d_x}{\sum_{x}\phi_i^2(x)\rho(x)d_x}
    \le \frac{4\sum_{x\in\Omega}\rho^{-1}(x)\Gamma(g,u_i)^2d_x}
    {\frac{1}{2n}\sum_{x}u_i(x)^2d_x-I_g(u_i)+\sum_{j}(\lambda_i-\lambda_j)a_{ij}^2},
\end{align}
or
\begin{align}\label{eq:3.17}
    (\lambda_{k+1}-\lambda_{i})\Bigl(\frac{1}{2n}\sum_{x}u_i(x)^2d_x-I_g(u_i)+\sum_{j}(\lambda_i-\lambda_j)a_{ij}^2\Bigr)\le 4\sum_{x\in\Omega}\rho^{-1}(x)\Gamma(g,u_i)^2d_x,
\end{align}
If we simply average \eqref{eq:3.13} and \eqref{eq:3.17} we obtain

  \begin{align}(\lambda_{k+1}-\lambda_{i})\Bigl(\frac{1}{2n}\sum_{x}u_i(x)^2d_x-I_g(u_i)\Bigr)+\sum_{j}^{k}\Bigl(\lambda_{k+1}-\frac{\lambda_i+\lambda_{j}}{2}\Bigr)(\lambda_i-\lambda_j)a_{ij}^2
    \notag\\\le 4\sum_{x\in \Omega}\rho^{-1}(x)\Gamma(g,u_i)^2d_x.\end{align}
Then,by the same procedure,we have
\begin{align*}
\lambda_{k+1}\Bigl(1-\frac{\rho_{max}}{k}\sum_{i}^{k}\lambda_i\Bigr)\le \frac{1}{k}\Bigl(1+\frac{4\rho_{max}}{n\rho_{min}}\Bigr)\sum_{i}^{k}\lambda_i-\frac{\rho_{max}}{k}\sum_{i}^{k}\lambda_i^2.
\end{align*}
This proves the Theorem \ref{thm:1.2}.
\end{proof}
\begin{proof}[Proof of Theorem 1.3]
By \eqref{eq:3.17}, we also have
\begin{align}
\frac{1}{2n}\sum_{x}u_i(x)^2d_x-I_g(u_i)+\sum_{j}(\lambda_i-\lambda_j)a_{ij}^2\le \frac{4\sum_{x\in\Omega}\rho^{-1}(x)\Gamma(g,u_i)^2d_x}
    {\lambda_{k+1}-\lambda_{i}}.
\end{align} 
And,by the same procedure above, we abtain HP inequality
\begin{align*}
\sum_{i}^{k}\frac{\lambda_i}{\lambda_{k+1}-\lambda_i}\ge \frac{nk}{4}\Bigl(\frac{\rho_{min}}{\rho_{max}}-\frac{\rho_{min}}{k}\sum_{i}^{k}\lambda_i\Bigr).
\end{align*}
Which proves theorem \ref{thm:1.3}.
\end{proof}
\begin{proof}[Proof of Theorem 1.4]
Change $\lambda_{k+1}-\lambda_i$ into $\lambda_{k+1}-\lambda_{k}$,with the same procedure,we deduce PPW inequality
\begin{align*}
(\lambda_{k+1}-\lambda_k)\Bigl(\frac{\rho_{min}}{\rho_{max}}-\frac{\rho_{min}}{k}\sum_{i}^{k}\lambda_i\Bigr)\le\frac{4}{nk}\sum_{i}^{k}\lambda_i.
\end{align*}
We prove the Theorem \ref{thm:1.4}.
\end{proof}
\bibliographystyle{plain}
\bibliography{schodinger}
\addcontentsline{toc}{section}{参考文献}
\end{document}